\documentclass[12pt]{amsart}

\usepackage{amsmath,amssymb,amsthm,color}
\definecolor{shadecolor}{rgb}{.61,.87,1}

\DeclareMathOperator{\Aff}{Aff}
\DeclareMathOperator{\Spec}{Spec}
\DeclareMathOperator{\Stab}{Stab}
\DeclareMathOperator{\GL}{GL}
\DeclareMathOperator{\SL}{SL}

\newcommand{\sm}{\setminus}

\newcommand{\C}{\mathbb{C}}

\newcommand{\N}{\mathbb{N}}
\newcommand{\Z}{\mathbb{Z}}

\newcommand{\F}{\mathbb{F}}

\newcommand{\I}{\mathcal{O}}
\newcommand{\ve}{\varepsilon}

\newcommand{\1}{^{-1}}

\newcommand{\f}[2]{\frac{#1}{#2}}

\newcommand{\nsg}{\trianglelefteq}

\newtheorem{theorem}{Theorem}[section]
\newtheorem{definition}[theorem]{Definition}

\newtheorem{proposition}[theorem]{Proposition}
\newtheorem{lemma}[theorem]{Lemma}

\numberwithin{theorem}{section}

\title{On exponential sums over orbits in $\F_p^d$}
\author{Sarah Peluse}
\address{Department of Mathematics, Stanford University, Stanford, California 94305}
\email{speluse@stanford.edu}

\begin{document}
\maketitle

\begin{abstract}
This paper proves a bound for exponential sums over orbits of vectors in $\F_p^d$ under subgroups of $\GL_d(\F_p)$. The main tool is a classification theorem for approximate groups due to Gill, Helfgott, Pyber, and Szab\'o.
\end{abstract}

\section{Introduction}
In~\cite{BGK}, Bourgain, Glibichuk, and Konyagin used methods from additive combinatorics to prove that, for every $\delta>0$, there exists an $\ve=\ve_\delta>0$ such that if $H\leq \F_p^\times$ and $|H|\geq p^\delta$, then
\begin{equation}\label{bgk}
\max_{\xi\in\F_p\sm\{0\}}\left|\sum_{x\in H} e_p(\xi x)\right|\leq p^{-\ve}|H|.
\end{equation}
Here $e_p(y):=e^{2\pi i y/p}$. Bounds similar to (\ref{bgk}) were then shown for exponential sums over multiplicative subgroups in $\Z/q\Z$ by Bourgain~\cite{B2}, in $\F_p\times\dots\times\F_p$ by Bourgain~\cite{B1}, and in general finite fields by Bourgain and Chang~\cite{BC}, and for sums of nontrivial additive characters over multiplicative subgroups in finite commutative rings by Bourgain~\cite{B3}. The purpose of this paper is to use recent progress on the classification of approximate groups in linear algebraic groups to prove the following general result, from which the bounds of~\cite{B1} and~\cite{BGK} can be recovered as special cases:
\begin{theorem}\label{MT}
For every $d\in\N$ and $\delta,\beta>0$, there exists an $\ve=\ve_{d,\delta,\beta}>0$ such that the following holds. If $H\leq\GL_d(\F_p)$, $v\in\F_p^d\sm\{0\}$, and the orbit $\I:=Hv$ satisfies
\begin{itemize}
\item $|\I|\geq p^\delta$ and
\item $|\I\cap P|\leq|\I|^{1-\beta}$ for every hyperplane $P\subset\F_p^d$,
\end{itemize}
then
\begin{equation}\label{p}
\max_{\xi\in\F_p^d\sm\{0\}}\left|\sum_{x\in\I}e_p(\xi\cdot x)\right|\leq p^{-\ve}|\I|.
\end{equation}
\end{theorem}
Clearly, some condition on the intersection of $\I$ with hyperplanes is necessary in Theorem~\ref{MT}. If we allowed $\I\cap\{x\in\F_p^d:\xi\cdot x=c\}$ to be large for some nonzero $\xi\in\F_p^d$, then the sum in~(\ref{p}) could be quite big, due to having many summands equal to $e_p(c)$.

The proofs of~\cite{B2},~\cite{B1},~\cite{B3},~\cite{BC}, and~\cite{BGK} all follow the same general strategy. If (\ref{bgk}) or the corresponding bound fails to hold, one can use the Balog-Szemer\'edi-Gowers Theorem to construct a subset $A$ of the ring $R$ that grows slowly under addition and multiplication. Various properties of this $A$ can be deduced from those of $H$; for example, a lower bound on $|H|$ tells us that $A$ cannot be almost all of $R$. In each of~\cite{B2},~\cite{B1},~\cite{B3},~\cite{BC}, and~\cite{BGK}, a sum-product theorem in $R$ is shown, which places restrictions on $A$. These restrictions end up being incompatible with other properties of $A$, however, and so (\ref{bgk}) must hold. The proof of Theorem~\ref{MT} proceeds in the spirit of~\cite{BGK} and its descendants, but here a classification result of Gill, Helfgott, Pyber, and Szab\'o~\cite{GH} plays the role of a sum-product theorem.

We now briefly review some results on growth in linear algebraic groups. See the surveys of Breuillard~\cite{Br}, Green~\cite{G}, and Helfgott~\cite{H3} for much more detailed descriptions of the history of this subject and its applications. The first major breakthrough was the work of Helfgott~\cite{H1} in $\SL_2(\F_p)$. Helfgott's result essentially says that every approximate group that is not contained in a proper subgroup of $\SL_2(\F_p)$ is either very small or is almost all of $\SL_2(\F_p)$. The same classification of approximate groups was shown in $\SL_2(\F_q)$ by Dinai~\cite{D} and in $\SL_3(\F_p)$ by Helfgott~\cite{H2}. Following this, the result was extended to finite simple groups of Lie type of bounded rank over any finite field independently by Pyber and Szab\'o~\cite{PS} and by Breuillard, Green, and Tao~\cite{BrGT1}. In~\cite{H2}, Helfgott also proved a classification result for all approximate groups in $\SL_3(\F_p)$, not just those that generate the whole group. More generally, the Helfgott-Lindenstrauss conjecture describes a qualitative classification of approximate groups contained in any group. This conjecture was proven by Breuillard, Green, and Tao~\cite{BrGT2}, and a quantitative version for subgroups of $\GL_d(\F_p)$ was shown by Helfgott, Gill, Pyber, and Szab\'o~\cite{GH}:
\begin{theorem}[Gill, Helfgott, Pyber, and Szab\'o~\cite{GH}, Theorem 2]\label{GHPS}
Let $A$ be a subset of $\GL_{d}(\F_p)$. Assume that $I\in A$ and $A=A\1$. Then for every $C\geq 1$, either $|A^3|\geq C|A|$ or there are subgroups $H_1\leq H_2$ in $\GL_d(\F_p)$ and an integer $k\ll_{d} 1$ such that
\begin{itemize}
\item $H_1,H_2\trianglelefteq \langle A\rangle$ and $H_2/H_1$ is nilpotent,
\item $A^k$ contains $H_1$, and
\item $|A^k\cap H_2|\geq C^{-O_{d}(1)}|A|$.
\end{itemize}
\end{theorem}
Theorem~\ref{GHPS} will be essential in the proof of Theorem~\ref{MT}.

This paper is organized as follows. In Section~\ref{prelims}, we recall some basic notions from arithmetic combinatorics. In Section~\ref{aprxgp}, we construct a family of approximate groups that are closely connected to $H$ and the orbit $\I$. This connection and the structure of $\Aff_d(\F_p)$ as a semidirect product will be key in Section~\ref{info}, where we derive several properties of such approximate groups using Theorem~\ref{GHPS}. In Section~\ref{proof}, we show that if (\ref{p}) does not hold, then there exists an approximate group with conflicting properties, proving Theorem~\ref{MT}.

\section*{Acknowledgments}
The author thanks Zeb Brady, Harald Helfgott, and Kannan Soundararajan for helpful comments on earlier drafts of this paper. This paper has its roots in a project from the 2016 Arizona Winter School, during which the author worked with Brendan Murphy, Vlad Petkov, and Lam Pham to use a result on growth in $\Aff_1(\F_p)$ from~\cite{H3} to prove the Bourgain-Glibichuk-Konyagin bound~\cite{MPPP}. The author is grateful to Harald Helfgott for suggesting that project. 

The author is supported by the National Science Foundation Graduate Research Fellowship Program under Grant No. DGE-114747 and by the Stanford University Mayfield Graduate Fellowship.

\section{Preliminaries}\label{prelims}
Let $G$ be a group and let $A,A'\subset G$ be nonempty subsets. We can then form the product set of $A$ and $A'$,
\[
AA'=\{aa':a\in A\text{ and }a'\in A'\},
\]
the restricted product of $A$ and $A'$ over a subset $E\subset A\times A'$,
\[
A\underset{E}{\cdot} A'=\{aa':(a,a')\in E\},
\]
and the set of inverses of elements in $A$,
\[
A\1=\{a\1:a\in A\}.
\]
For each $m\in\N$, denote the $m$-fold product of $A$ with itself by $A^m$. Approximate groups are a particular type of set that grows slowly under taking products with itself:
\begin{definition}[Tao~\cite{T1}]
Let $G$ be a group and $K\geq 1$. We say that $A\subset G$ is a {\bf $K$-approximate group} if $1\in A$, $A=A\1$, and there exists an $X\subset G$ with $|X|\leq K$ such that $A^2\subset XA$.
\end{definition}
It follows immediately from the definition that if $A$ is a $K$-approximate group and $m\in\N$, then $A^m\subset X^{m-1}A$, and thus $|A^m|\leq K^{m-1}|A|$. In addition, $A^{2m}\subset X^{m}A^m$, so that $A^m$ is a $K^m$-approximate group. The Balog-Szemer\'edi-Gowers Theorem, which in the noncommutative setting is due to Tao~\cite{T1}, tells us that if a set grows slowly, then a large part of it is contained in a coset of an approximate group:
\begin{theorem}[Balog-Szemer\'edi-Gowers]
Let $G$ be any group and let $A$ be a finite subset of $G$. If $K\geq 1$ and there exists a subset $E\subset A\times A$ such that
\[
|E|\geq K^{-1}|A|^2 \text{ and } |A\underset{E}{\cdot}A|\leq K|A|,
\]
then there exists a $g\in G$ and a $O(K^{O(1)})$-approximate group $B\subset G$ such that $|B|\ll K^{O(1)}|A|$ and $|A\cap gB|\gg K^{-O(1)}|A|$.
\end{theorem}
There are many ways of formulating this result; see Theorem 5.4 of~\cite{T1} or Section 2.7 of~\cite{TV}. 

We will also need a small amount of background on Fourier analysis on $\F_p^d$. For every $0\leq \alpha\leq 1$ and $A\subset\F_p^d$, set
\[
\Spec_\alpha(A):=\left\{\xi\in\F_p^d:\left|\f{1}{|A|}\sum_{x\in A}e_p(\xi\cdot x)\right|>\alpha\right\}.
\]
Rephrased using this notation, the goal of this paper is to prove that there exists an $\ve=\ve_{d,\delta,\beta}>0$ such that for all primes $p$ and all $\I$ as in Theorem~\ref{MT}, $\Spec_{p^{-\ve}}(\I)=\{0\}$. Clearly, $\Spec_{\alpha}(A)$ is symmetric and contains $0$. The following lemma, due to Bourgain, says that $\Spec_{\alpha}(A)$ is also approximately closed under taking differences:
\begin{lemma}\label{ha}
Let $A\subset\F_p^d$. For every $0\leq \alpha\leq 1$, we have
\[
|\{(\xi_1,\xi_2)\in \Spec_{\alpha}(A)\times \Spec_{\alpha}(A):\xi_1-\xi_2\in\Spec_{\alpha^2/2}(A)\}|\geq\f{\alpha^2}{2}|\Spec_{\alpha}(A)|^2.
\]
\end{lemma}
See Section 4.6 of~\cite{TV} for a proof.

\section{A family of approximate groups}\label{aprxgp}
Suppose that $H\leq\GL_d(\F_p)$ and $v\in\F_p^d\sm\{0\}$. For ease of notation, we would like for the action of $H$, not $H^T$, on the $\xi$ to preserve the value of the sum in (\ref{p}). Thus, we will instead prove (\ref{p}) for the orbit $\I:=H^Tv=\{M^Tv:M\in H\}$, assuming it satisfies the conditions stated in Theorem~\ref{MT}. This means that $\Spec_{\alpha}(\I)$ is preserved by the action of $H$ for each $\alpha$, a fact that will be important in the proof of Proposition~\ref{P} below, and also later in Section~\ref{info}.

For each $\alpha$, embed the set $H\times\Spec_\alpha(\I)$ in $G:=\Aff_d(\F_p)$ as
\[
A_{\alpha}:=\begin{pmatrix}
H & \Spec_\alpha(\I) \\
 & 1
\end{pmatrix}.
\]
Note that, since $\Spec_\alpha(\I)$ is symmetric and $H$-invariant, $A_\alpha=A_{\alpha}\1$ for all $\alpha$. In the following proposition, we show that $A_\alpha$ grows slowly for certain choices of $\alpha$, and thus has large intersection with a coset of an approximate group. Our method of doing this is essentially identical to the one used in~\cite{MPPP}, which worked in $\Aff_1(\F_p)$. The proof uses an iteration and pigeonholing argument that was inspired by the proof of the Bourgain-Glibichuk-Konyagin bound given in~\cite{TV}.
\begin{proposition}\label{P}
For every $\ve'>0$, there exists a $g\in G$ and a $O_{d,\ve'}(p^{O(\ve')})$-approximate group $B\subset G$ satisfying
\[
|B|\ll_{d,\ve'}p^{O(\ve')}|A_{\phi_d(\ve')}|
\]
and
\[
|A_{\phi_d(\ve')}\cap gB|\gg_{d,\ve'} p^{-O(\ve')}|A_{\phi_d(\ve')}|,
\]
where $\phi_d(\ve')=2^{1-c_{d,\ve'}}p^{-\ve'c_{d,\ve'}/2^{\lceil 2d/\ve'\rceil}}$ for some $1\leq c_{d,\ve'}\leq 2^{\lceil 2d/\ve'\rceil}$.
\end{proposition}
\begin{proof}
Let $J\in\N$ and $\ve_0>0$, to be chosen later. Set $\alpha_0=p^{-\ve_0}$ and $\alpha_j=\f{\alpha_{j-1}^2}{2}$ for each $1\leq j\leq J$, so $\alpha_j=2^{1-2^j}p^{-2^j\ve_0}$. We will first show that, for every $0\leq j\leq J-1$, there exists an $E_j\subset A_{\alpha_j}\times A_{\alpha_j}$ such that
\[
|E_j|\geq\f{\alpha_j^2}{2}|A_{\alpha_j}|^2\text{ and }A_{\alpha_j}\underset{E_j}{\cdot}A_{\alpha_j}\subset A_{\alpha_{j+1}}.
\]
Apply Lemma~\ref{ha} to $\Spec_{\alpha_j}(\I)$. This gives us a subset $F_j$ of $\Spec_{\alpha_j}(\I)\times\Spec_{\alpha_j}(\I)$ such that
\[
\Spec_{\alpha_j}(\I)\underset{F_j}{-}\Spec_{\alpha_j}(\I)\subset\Spec_{\alpha_{j+1}}(\I)
\]
and $|F_j|\geq\f{\alpha_j^2}{2}|\Spec_{\alpha_j}(\I)|^2$. For every $(\xi_1,\xi_2)\in F_j$, define
\[
E_{(\xi_1,\xi_2)}=\left\{\left(\begin{pmatrix} M_1 & \xi_1 \\
& 1 
\end{pmatrix},
\begin{pmatrix} M_2 & -M_1\1\xi_2 \\
& 1
\end{pmatrix}\right): M_1,M_2\in H\right\},
\]
so that if $(N_1,N_2)\in E_{(\xi_1,\xi_2)}$, then $N_1N_2\in A_{\alpha_{j+1}}$. Since $\Spec_{\alpha_j}(\I)$ is symmetric and $H$-invariant, $E_{(\xi_1,\xi_2)}\subset A_{\alpha_j}\times A_{\alpha_j}$. Note that we can recover $(\xi_1,\xi_2)$ by knowing any element of $E_{(\xi_1,\xi_2)}$, so for $(\xi_1,\xi_2)\neq(\eta_1,\eta_2)$, the sets $E_{(\xi_1,\xi_2)}$ and $E_{(\eta_1,\eta_2)}$ are disjoint. Take
\[
E_j=\coprod_{(\xi_1,\xi_2)\in F_j}E_{(\xi_1,\xi_2)}.
\]
Then
\[
|E_j|=|F_j||H|^2\geq\f{\alpha_j^2}{2}|\Spec_{\alpha_j}(\I)|^2|H|^2=\f{\alpha_j^2}{2}|A_{\alpha_j}|^2,
\]
and since $A_{\alpha_j}\underset{E_j}{\cdot} A_{\alpha_j}\subset A_{\alpha_{j+1}}$ by construction, we have our desired set $E_j$ for each $0\leq j\leq J-1$.

Now, because $\alpha_j$ decreases as $j$ increases, $|\Spec_{\alpha_j}(\I)|$ increases as $j$ increases. Thus, because $|\Spec_{\alpha}(\I)|\leq p^d$ for all $\alpha$, there must exist a $0\leq j\leq J-1$ such that $|\Spec_{\alpha_{j+1}}(\I)|\leq p^{d/J}|\Spec_{\alpha_j}(\I)|$. As a consequence, for this $j$ we have the bounds
\[
|E_j|\geq\f{\alpha_j^2}{2}|A_{\alpha_j}|^2\text{ and }|A_{\alpha_j}\underset{E_j}{\cdot} A_{\alpha_j}|\leq p^{d/J}|A_{\alpha_j}|.
\]
Apply the Balog-Szemer\'edi-Gowers Theorem to $A_{\alpha_j}$ with $K=\f{2}{\alpha_j^2}p^{d/J}$. This yields a $g\in G$ and a $O((2^{2^j}p^{2^j\ve_0+d/J})^{O(1)})$-approximate group $B\subset G$ such that $|B|\ll(2^{2^j}p^{2^j\ve_0+d/J})^{O(1)}|A_{\alpha_j}|$ and $|A_{\alpha_j}\cap gB|\gg (2^{2^j}p^{2^j\ve_0+d/J})^{-O(1)}|A_{\alpha_j}|$, since $\alpha_j=2^{1-2^j}p^{-2^j\ve_0}$. Now take $J=\lceil\f{2d}{\ve'}\rceil$ and $\ve_0=\f{\ve'}{2^J}$. So, setting $\phi_d(\ve'):=\alpha_j=2^{1-c_{d,\ve'}}p^{-c_{d,\ve'}\ve_0}$ where $c_{d,\ve'}=2^j\in[0,2^J]$, we have that $B$ is a $O_{d,\ve'}(p^{O(\ve')})$-approximate group, $|B|\ll_{d,\ve'}p^{O(\ve')}|A_{\phi_d(\ve')}|$, and $|A_{\phi_d(\ve')}\cap gB|\gg_{d,\ve'}p^{-O(\ve')}|A_{\phi_d(\ve')}|$.
\end{proof}

\section{Properties of the approximate groups}\label{info}
For any subset $A\subset G$, set
\[
L(A):=\left\{M\in\GL_d(\F_p):\begin{pmatrix} M & \xi \\ & 1 \end{pmatrix}\in A\text{ for some }\xi\in\F_p^d\right\},
\]
and for any $M\in L(A)$, set 
\[
R_M(A):=\left\{\xi\in\F_p^d:\begin{pmatrix} M & \xi \\ & 1 \end{pmatrix}\in A\right\}.
\]
Similarly, for any $g\in G$, write $L(g)$ for the upper left-hand $d\times d$ block of $g$ and $R(g)$ for the upper right-hand $d\times 1$ block of $g$. In this section, we use Theorem~\ref{GHPS} and the structure of $G$ as a semidirect product to show that any $O_{d,\ve'}(p^{O(\ve')})$-approximate group $B$ with a coset that is close to $A_{\phi_d(\ve')}$ must be of a certain form. There exists a $D\ll_d 1$ such that whenever $\Spec_{\phi_d(\ve')}(\I)\neq\{0\}$, $\ve'$ is small enough, $p$ is large enough, and $M\in L(B)$, we have that $R_M(B^{D})$ is almost equal to a disjoint union of translates of some nontrivial subspace of $\F_p^d$. At the same time, each $R_M(B^{D})$ is close to the image of $\Spec_{\phi_d(\ve')}(\I)$ under a linear map. These two facts will be combined to derive a contradiction in Section~\ref{proof}.

We will repeatedly use two simple observations about $L$ and $R$. The first is that, if $A,A'\subset G$ and $g\in A\cap A'$, then $L(g)\in L(A\cap A')$ and $R(g)\in R_{L(g)}(A\cap A')$. Thus,
\[
|A\cap A'|\leq|L(A\cap A')|\max_{N\in L(A\cap A')}|R_N(A\cap A')|.
\]
The second observation is that there exist at least $|L(A)|$ elements $M\in L(A^2)$ for which $|R_M(A^2)|\geq\max_{N\in L(A)}|R_N(A)|$. This is because, if $|R_{N_0}(A)|=\max_{N\in L(A)}|R_N(A)|$ and $N_1,N_2\in L(A)$ are distinct, then
\[
\begin{pmatrix} N_0 & R_{N_0}(A) \\ & 1 \end{pmatrix}\cdot\begin{pmatrix} N_1 & R_{N_1}(A) \\ & 1 \end{pmatrix}=\begin{pmatrix} N_0N_1 & N_0R_{N_1}(A)+R_{N_0}(A) \\ & 1\end{pmatrix}
\]
and
\[
\begin{pmatrix} N_0 & R_{N_0}(A) \\ & 1 \end{pmatrix}\cdot\begin{pmatrix} N_2 & R_{N_2}(A) \\ & 1 \end{pmatrix}=\begin{pmatrix} N_0N_2 & N_0R_{N_2}(A)+R_{N_0}(A) \\ & 1\end{pmatrix}
\]
are disjoint subsets of $A^2$, each of cardinality at least $|R_{N_0}(A)|$. Thus, we have that
\[
|A^2|\geq |L(A)|\max_{N\in L(A)}|R_N(A)|.
\]
Similarly, if $A=A\1$, then $|R_M(A^3)|\geq\max_{N\in L(A)}|R_N(A)|$ for all $M\in L(A)$. This is because the sets
\[
\begin{pmatrix} M & R_{M}(A) \\ & 1 \end{pmatrix}\cdot\begin{pmatrix} N_0 & R_{N_0}(A) \\ & 1 \end{pmatrix}\cdot\begin{pmatrix} N_0\1 & R_{N_0\1}(A) \\ & 1 \end{pmatrix},
\]
which are contained in $A^3$, have size at least $\max_{N\in L(A)}|R_N(A)|$ for every $M\in L(A)$.

The following lemma records some immediate consequences of these observations when an approximate group is involved.
\begin{lemma}\label{blocks}
Let $A,B\subset G$. Suppose that $B$ is a $K$-approximate group such that $|B|\leq K|A|$ and $|A\cap B|\geq K\1 |A|$. Then,
\[
|L(B)|\leq K^3|L(A)\cap L(B)|,
\]
\[
|L(A)\cap L(B)|\geq K\1\f{\min_{N\in L(A)}|R_N(A)|}{\max_{N\in L(A)}|R_N(A)|}|L(A)|,
\]
\[
\max_{N\in L(B)}|R_N(B)|\leq K^3\max_{N\in L(A\cap B)}|R_N(A)\cap R_N(B)|,
\]
and
\[
\max_{N\in L(A\cap B)}|R_N(A)\cap R_N(B)|\geq K\1\min_{N\in L(A)}|R_N(A)|.
\]
\end{lemma}
\begin{proof}
By the remarks above, we have
\[
|A\cap B|\leq |L(A)\cap L(B)|\max_{N\in L(A\cap B)}|R_N(A)\cap R_N(B)|,
\]
\[
|A\cap B|\geq K\1|A|\geq K\1 |L(A)|\min_{N\in L(A)}|R_N(A)|,
\]
and
\[
|A\cap B|\geq K^{-3}|B^2|\geq K^{-3}|L(B)|\max_{N\in L(B)}|R_N(B)|,
\]
since $|A|\geq K\1 |B|\geq K^{-2}|B^2|$. The four inequalities follow by comparing the upper and lower bounds for $|A\cap B|$.
\end{proof}

Note that the conditions on $\I$ listed in Theorem~\ref{MT} imply that, for all $\xi\in\F_p^d\sm\{0\}$, we have $|\Stab_{H}(\xi)|\leq p^{-\delta\beta}|H|$ and, equivalently, $|H\xi|\geq p^{\delta\beta}$. Indeed, if $N\in \Stab_H(\xi)$, then $\xi\cdot N^Tv=\xi\cdot v$, so that $N^Tv$ is contained in the hyperplane $P=\{x\in\F_p^d:\xi\cdot x=\xi\cdot v\}$. Thus,
\[
|\Stab_H(\xi)|\leq|\Stab_{H^T}(v)||\I\cap P|\leq|\Stab_{H^T}(v)||\I|^{1-\beta}=|\I|^{-\beta}|H|.
\]
Since $|\I|\geq p^{\delta}$, this is bounded above by $p^{-\delta\beta}|H|$. The next lemma shows that if $B$ is an approximate group that is close to a translate of $A_\alpha$, then we also have a similar upper bound for $|L(B)\cap\Stab(\xi)|$:

\begin{lemma}\label{stab}
Suppose that $B\subset G$ is a $K$-approximate group satisfying $|B|\leq K|A_\alpha|$ and $|A_\alpha\cap gB|\geq K\1 |A_\alpha|$ for some $g\in G$. Then for all $\xi\in\F_p^d\sm\{0\}$, we have
\[
|L(B)\cap\Stab(\xi)|\leq p^{-\delta\beta}K^7|L(B)|.
\]
\end{lemma}
\begin{proof}
First, we will show that at least $p^{\delta\beta}K^{-1}$ cosets of $\Stab(\xi)$ have some representative in $L(B)$. Let $S\subset L(B)$ contain exactly one representative from each coset of $\Stab(\xi)$ that intersects $L(B)$. By the remarks before the statement of the lemma, we have
\begin{align*}
|L(g)\1 H\cap L(B)| &= \sum_{M\in S}|L(g)\1 H\cap M\Stab(\xi)\cap L(B)| \\
&\leq \sum_{M\in S}|H\cap L(g)M\Stab(\xi)| \\
&\leq p^{-\delta\beta}|H||S|.
\end{align*}
On the other hand, by Lemma~\ref{blocks} applied to $g\1 A_\alpha$ and $B$, we also have that $|L(g)\1 H\cap L(B)|\geq K\1 |H|$ since $|L(g\1 A_\alpha)|=|H|$ and $|R_M(g\1 A_\alpha)|=|\Spec_{\alpha}(\I)|$ for all $M\in L(g\1 A_\alpha)$. Comparing the upper and lower bounds for $|L(g)\1 H\cap L(B)|$ yields $|S|\geq p^{\delta\beta}K\1$.

Now, for each coset $M\Stab(\xi)$ intersecting $L(B)$, we have
\[
L(B^2)\cap M\Stab(\xi)\supset [L(B)\cap M\Stab(\xi)][L(B)\cap\Stab(\xi)],
\]
and as a consequence,
\[
|L(B^2)\cap M\Stab(\xi)|\geq |L(B)\cap\Stab(\xi)|.
\]
Since there are at least $p^{\delta\beta}K\1$ distinct cosets of $\Stab(\xi)$ with a representative in $L(B)$, this implies that
\[
|L(B^2)|\geq p^{\delta\beta}K\1|L(B)\cap\Stab(\xi)|.
\]
By Lemma~\ref{blocks} applied to $B$ and $B^2$, which is a $K^2$-approximate group, we also have $|L(B^2)|\leq K^6|L(B)|$. Combining this with the lower bound for $|L(B^2)|$ completes the proof of the lemma.
\end{proof}
One of the key ideas from~\cite{BGK} is that, if $\Spec_\alpha(\I)$ contains one nonzero frequency $\xi$, then it must contain at least $|H|=|H\xi|$ nonzero frequencies by the $H$-invariance of $\Spec_\alpha(\I)$. We, too, will use this fact in the proof of the following proposition, which is the main result of this section:
\begin{proposition}\label{main}
There exists a $k\ll_d 1$ such that the following holds. Let $1\gg_{d,\delta,\beta}\ve'>0$ and $p\gg_{d,\delta,\beta} 1$. Suppose that $B\subset G$ is a $O_{d,\ve'}(p^{O(\ve')})$-approximate group satisfying
\[
|B|\ll_{d,\ve'}p^{O(\ve')}|A_{\phi_d(\ve')}|
\]
and
\[
|A_{\phi_d(\ve')}\cap gB|\gg_{d,\ve'}p^{-O(\ve')}|A_{\phi_d(\ve')}|
\]
for some $g\in G$. If $\Spec_{\phi_d(\ve')}(\I)\neq\{0\}$, then for every $M\in L(B)$, there exists a $T_M\subset R_M(B^{4k})$ equal to a union of cosets of some nontrivial subspace of $\F_p^d$ such that $R_M(B^{3k})\subset T_M$ and $|R_M(B^{3k})|\gg_{d,\ve'}p^{-O_d(\ve')}|T_M|$.
\end{proposition}
\begin{proof}
Apply Theorem~\ref{GHPS} to $B$. This yields subgroups $H_1,H_2\nsg\langle B\rangle$ such that $H_1\leq H_2$, $H_2/H_1$ is nilpotent, $B^k\supset H_1$, and $|B^k\cap H_2|\gg_{d,\ve'}p^{-O_d(\ve')}|B|\gg_{d,\ve'}p^{-O_d(\ve')}|A_{\phi_d(\ve')}|$, where $k\ll_d 1$. We now show that $R_I(H_1)$ must contain a nontrivial subspace of $\F_p^d$ if $\ve'$ is sufficiently small and $p$ is sufficiently large in terms of $d,\delta,$ and $\beta$. So, set
\[
G':=\begin{pmatrix} I & \F_p^d \\ & 1\end{pmatrix}\subset G,
\]
and assume that $\Spec_{\phi_d(\ve')}(\I)\neq\{0\}$ and $H_1\cap G'$ is trivial.

First, we will show that $H_2$ contains a nonidentity element of $G'$ when $\ve'$ is sufficiently small and $p$ is sufficiently large. If $H_2\cap G'$ were trivial, then clearly $|R_M(H_2)|\leq 1$ for every $M\in \GL_d(\F_p)$ since $H_2$ is a subgroup. This would imply that
\[
|B^k\cap H_2|\leq |L(B^k)|\ll_{d,\ve'}p^{O(\ve')}|H|
\]
by Lemma~\ref{blocks}. On the other hand, since $|A_{\phi_d(\ve')}|=|H||\Spec_{\phi_d(\ve')}(\I)|$, we have $|B^k\cap H_2|\gg_{d,\ve'}p^{-O_d(\ve')}|H||\Spec_{\phi_d(\ve')}(\I)|$. Comparing the upper and lower bounds for $|B^k\cap H_2|$ yields $|\Spec_{\phi_d(\ve')}(\I)|\ll_{d,\ve'}p^{O_d(\ve')}$. However, we must have $|\Spec_{\phi_d(\ve')}(\I)|\geq p^{\delta\beta}$, for $\Spec_{\phi_d(\ve')}(\I)$ is $H$-invariant and contains a nonzero element $\xi$, and thus contains every element of $H\xi$. This implies that $p^{\delta\beta}\ll_{d,\ve'}p^{O_d(\ve')}$, which is impossible for $\ve'\ll_{d,\delta,\beta} 1$ and $p\gg_{d,\delta,\beta}1$. 

Next, we will show that, when $\ve'$ is sufficiently small and $p$ is sufficiently large, there must exist an $M\in L(H_2)\sm\Stab(\xi)$ for every $\xi\in\F_p^d\sm\{0\}$. If there were no such $M$ for some nonzero $\xi$, then we would have $L(H_2)\subset\Stab(\xi)$, and thus
\[
|B^k\cap H_2|\leq |L(B^k)\cap\Stab(\xi)|\max_{N\in L(B^k)}|R_N(B^k)|.
\]
By Lemma~\ref{stab}, the above is $\ll_{d,\ve'}p^{-\delta\beta+O_d(\ve')}|L(B^k)|\max_{N\in L(B^k)}|R_N(B^k)|$. On the other hand,
\[
|B^k\cap H_2|\gg_{d,\ve'}p^{-O_d(\ve')}|B^{2k}|\gg_{d,\ve'}p^{-O_d(\ve')}|L(B^k)|\max_{N\in L(B^k)}|R_N(B^k)|,
\]
since $B$ is a $O_{d,\ve'}(p^{O(\ve')})$-approximate group. Comparing the upper and lower bounds for $|B^k\cap H_2|$ again yields the inequality $p^{\delta\beta}\ll_{d,\ve'}p^{O_d(\ve')}$. For $\ve'\ll_{d,\delta,\beta} 1$ and $p\gg_{d,\delta,\beta} 1$, this is impossible.

Now assume that $\ve'$ is sufficiently small and $p$ is sufficiently large so that $H_2\cap G'$ contains a nonidentity element and $L(H_2)\sm\Stab(\xi)$ is nonempty for every $\xi\in\F_p^d\sm\{0\}$. Then there exist
\[
\begin{pmatrix} I & \xi \\ & 1 \end{pmatrix},\begin{pmatrix} N & \eta \\ & 1\end{pmatrix}\in H_2
\]
such that $\xi\neq 0$ and $N\xi\neq\xi$. The commutator of these two elements is
\[
\left[\begin{pmatrix} I & \xi \\ & 1 \end{pmatrix},\begin{pmatrix} N & \eta \\ & 1\end{pmatrix}\right]=\begin{pmatrix} I & (I-N)\xi \\ & 1\end{pmatrix}.
\]
Thus, since $(I-N)\xi\neq 0$, the group $[H_2,H_2]$ contains a nonidentity element of $G'$. There also exists an $N'\in L(H_2)$ such that $(I-N')(I-N)\xi\neq 0$, so that $[H_2,[H_2,H_2]]$ contains a nonidentity element of $G'$ for the same reason that $[H_2,H_2]$ does. Repeating this argument shows that every term in the lower central series for $H_2$ contains a nonidentity element of $G'$. Taking the quotient by $H_1$, which we assumed has trivial intersection with $G'$, we see that the lower central series for $H_2/H_1$ never terminates. This contradicts the fact that $H_2/H_1$ is nilpotent.

Hence, for $1\gg_{d,\delta,\beta}\ve'>0$ and $p\gg_{d,\delta,\beta} 1$, if $\Spec_{\phi_d(\ve')}(\I)\neq\{0\}$, then $R_I(H_1)\neq\{0\}$, and in fact $R_I(H_1)$ contains an entire nontrivial subspace $V$ of $\F_p^d$ since $H_1$ is a subgroup. Because $B^k\supset H_1$, we have that $R_I(B^k)$ contains $V$ as well.

Let $M\in L(B)$ and set $T_M:= R_M(B^{3k})+V\subset R_M(B^{4k})$. Because $k\ll_d 1$, Lemma~\ref{blocks} applied to $B^{4k}$ and $B$ tells us that
\[
\max_{N\in L(B^4k)}|R_N(B^{4k})|\ll_{d,\ve'}p^{O_d(\ve')}\max_{N\in L(B)}|R_N(B)|.
\]
We also have $\max_{N\in L(B)}|R_N(B)|\leq |R_M(B^{3k})|$ since $B=B\1$. Thus,
\[
|T_M|\leq\max_{N\in L(B^{4k})}|R_N(B^{4k})|\ll_{d,\ve'}p^{O_d(\ve')}|R_M(B^{3k})|,
\]
and clearly $R_M(B^{3k})\subset T_M$. This completes the proof of the proposition.
\end{proof}

\section{Proof of Theorem~\ref{MT}}\label{proof}

In order to derive a contradiction when $\Spec_{\phi_d(\ve')}(\I)\neq\{0\}$, we will need one last lemma.

\begin{lemma}\label{spec}
For every $0<\alpha\leq 1$, if $\eta\in\F_p^d$ and $V\leq\F_p^d$ is a nontrivial subspace, then
\[
|\Spec_{\alpha}(\I)\cap(\eta+V)|\leq\alpha^{-2}p^{-\delta\beta}|V|.
\]
\end{lemma}
\begin{proof}
Define $f:V\to\C$ by
\[
f(v)=\f{1}{|\I|}\sum_{x\in\I}e_p(x\cdot(\eta+v))
\]
for all $v\in V$. Then
\[
\sum_{v\in V}|f(v)|^2\geq\alpha^2|\Spec_\alpha(\I)\cap (\eta+V)|.
\]
We have, for each $\xi\in\F_p^d/V^\perp$,
\begin{align*}
\hat{f}(\xi) &= \f{1}{|V|}\sum_{v\in V}\left(\f{1}{|\I|}\sum_{x\in\I}e_p(x\cdot(\eta+v))\right)e_p(\xi\cdot v) \\
&= \f{1}{|\I||V|}\sum_{x\in\I}e_p(x\cdot\eta)\sum_{v\in V}e_p((x+\xi)\cdot v) \\
&= \f{1}{|\I|}\sum_{x\in\I}e_p(x\cdot\eta)1_{V^\perp}(x+\xi).
\end{align*}
Hence,
\begin{align*}
\sum_{\xi\in\F_p^d/V^\perp}|\hat{f}(\xi)|^2 &= \f{1}{|\I|^2}\sum_{x,y\in\I}e_p((x-y)\cdot\eta)\sum_{\xi\in\F_p^d/V^\perp}1_{V^\perp}(x+\xi)1_{V^\perp}(y+\xi) \\
&= \f{1}{|\I|^2}\sum_{x,y\in\I}e_p((x-y)\cdot\eta)1_{V^\perp}(x-y) \\
&\leq \f{1}{|\I|^2}\sum_{x\in\I}|\{y\in\I:y\in x+V^\perp\}| \\
&\leq |\I|^{-\beta}
\end{align*}
by second condition on $\I$ in the statement of Theorem~\ref{MT}, because, as $\dim V>0$, the set $\{y\in\I:y\in x+V^\perp\}$ is contained in the intersection of $\I$ with a hyperplane. By Parseval's identity, it thus follows that
\[
\alpha^2|\Spec_\alpha(\I)\cap (\eta+V)| \leq \sum_{v\in V}|f(v)|^2=|V|\sum_{\xi\in\F_p^d/V^\perp}|\hat{f}(\xi)|^2\leq p^{-\delta\beta}|V|,
\]
since $|\I|\geq p^{\delta}$.
\end{proof}
Assume that $\ve'>0$ is sufficiently small and $p$ is sufficiently large so that Proposition~\ref{main} holds. Let $B$ be the $O_{d,\ve'}(p^{O(\ve')})$-approximate group given to us by Proposition~\ref{P}. For each $M\in L(B)$, write
\[
T_M=\coprod_{\xi\in F}(\xi+V)
\]
where $F\subset\F_p^d$ and $V\leq\F_p^d$ with $\dim V>0$. Then by Lemma~\ref{spec}, we have
\begin{align*}
|L(g)\1\Spec_{\phi_d(\ve')}(\I)\cap T_M| &= \sum_{\xi\in F}|\Spec_{\phi_d(\ve')}(\I)\cap L(g)(\xi+V)| \\ 
&\leq \phi_d(\ve')^{-2}p^{-\delta\beta}|F||V|.
\end{align*}
Thus, $|L(g)\1\Spec_{\phi_d(\ve')}(\I)\cap T_M|\leq \phi_d(\ve')^{-2}p^{-\delta\beta}|T_M|$. Applying Lemma~\ref{blocks} to $B$ and $B^{4k}$ yields
\[
\max_{N\in L(B)}|T_N|\leq\max_{N\in L(B^{4k})}|R_N(B^{4k})|\ll_{d,\ve'}p^{O_d(\ve')}\max_{N\in L(B)}|R_N(B)|,
\]
and applying Lemma~\ref{blocks} to $g\1 A_{\phi_d(\ve')}$ and $B$ yields
\[
\max_{N\in L(B)}|R_N(B)|\ll_{d,\ve'}p^{O(\ve')}\max_{N\in L(B)}|L(g)\1\Spec_{\phi_d(\ve')}(\I)\cap T_N|,
\]
since $R_M(B)\subset R_M(B^{3k})\subset T_M$ for all $M\in L(B)$. However, because $|L(g)\1\Spec_{\phi_d(\ve')}(\I)\cap T_M|\leq \phi_d(\ve')^{-2}p^{-\delta\beta}|T_M|$ for all $M\in L(B)$, we thus have the bound
\[
\max_{N\in L(B)}|T_N|\ll_{d,\ve'}\phi_d(\ve')^{-2}p^{-\delta\beta+O_d(\ve')}\max_{N\in L(B)}|T_N|.
\]
Since $\phi_d(\ve')^{-1}\ll_{d,\ve'}p^{\ve'}$, this yields the inequality $p^{\delta\beta}\ll_{d,\ve'}p^{O_d(\ve')}$, which is impossible for $1\gg_{d,\delta,\beta} \ve'>0$ and $p\gg_{d,\delta,\beta} 1$. Thus, there exists an $\ve'>0$ depending only on $d,\delta,$ and $\beta$ such that $\Spec_{\phi_d(\ve')}(\I)=\{0\}$ for all $p$ sufficiently large. As $\phi_d(\ve')=2^{1-c_{d,\ve'}}p^{-\ve'c_{d,\ve'}/2^{\lceil 2d/\ve'\rceil}}$ for some $1\leq c_{d,\ve'}\leq 2^{\lceil 2d/\ve'\rceil}$, this means that there exists an $\ve''>0$ (depending only on $d,\delta,$ and $\beta$) such that $\Spec_{p^{-\ve''}}(\I)=\{0\}$ for $p$ large enough (depending only on $d,\delta,$ and $\beta$). The bound (\ref{p}) is trivial for $p$ smaller than a fixed constant depending at most on $d,\delta,$ and $\beta$, so this clearly implies Theorem~\ref{MT}.
\bibliographystyle{plain}
\bibliography{exp}

\end{document}